\documentclass[11pt]{amsart}
\usepackage{amssymb, latexsym}
\usepackage{graphics, graphicx}

\newtheorem{theorem}{Theorem}
\newtheorem{lemma}{Lemma}
\newtheorem{corollary}{Corollary}
\newtheorem{example}{Example}
\newtheorem{definition}{Definition}
\newtheorem*{conjecture}{Conjecture}

\begin{document}
\renewcommand{\thepage}{}
{\bfseries\Large
\begin{center}
THE CANTOR GAME: WINNING\\
STRATEGIES AND DETERMINACY
\end{center}
}
\vspace{0.5in}
\begin{center}
by
\end{center}
\vspace{0.5in}
{\bfseries\Large
\begin{center}
MAGNUS D. LADUE
\end{center}
}

\newpage

\setcounter{section}{-1}

\section{Abstract}\label{S:abstract}
In \cite{GT98} Grossman and Turett define the Cantor game.  In \cite{mB07} Matt Baker
proves several results about the Cantor game and poses three challenging questions about
it:

\emph{Do there exist uncountable subsets of $[0, 1]$ for which:
\begin{enumerate}
\item Alice does not have a winning strategy;
\item Bob has a winning strategy;
\item neither Alice nor Bob has a winning strategy?
\end{enumerate}}

In this paper we show that the answers to these questions depend upon which axioms of set
theory are assumed.  Specifically, if we assume the Axiom of Determinacy in addition to the
Zermelo-Fraenkel axioms, then the answer to all three questions is ``no.''  If instead we
assume the Zermelo-Fraenkel axioms together with the Axiom of Choice, then the answer to
questions 1 and 3 is ``yes,'' and the answer to question 2 is likely to be ``no.''

Author's Note: This paper was my entry in the 2017 Regeneron Science Talent Search.  It
earned a Top 300 Scholar Award as well as Research Report and Student Initiative badges.

\newpage

\setcounter{page}{1}

\renewcommand{\thepage}{\arabic{page}}

\section{Introduction}\label{S:intro}
In their paper \cite{GT98} Grossman and Turett define the Cantor game.  In \cite{mB07} Matt
Baker proves several results about the Cantor game and poses several challenging questions
about it.  The Cantor game is an infinite game played on the real line by two players, A
(Alice) and B (Bob).  Let $a_{0}$ and $b_{0}$ be fixed real numbers such that $a_{0} <
b_{0}$, and in the interval $[a_{0}, b_{0}]$ fix a subset $S$, which we call the
\emph{target set}.  (In the original version of the Cantor game, $a_{0}$ and $b_{0}$ are
chosen to be 0 and 1, respectively.  In this case we take them to be arbitrary real
numbers.)  Initially, player A chooses $a_{0}$, and then player B chooses $b_{0}$.  On the
$n$th turn, for $n \ge 1$, player A chooses $a_{n}$ with the property that $a_{n-1} <
a_{n} < b_{n-1}$, and then player B chooses $b_{n}$ such that $a_{n} < b_{n} < b_{n-1}$.
Since the sequence $(a_{n})_{n \ge 0}$ is strictly increasing and bounded above by $b_{0}$,
$\lim_{n \to \infty} a_{n} = a$ exists, and since $(b_{n})_{n \ge 0}$ is strictly
decreasing and bounded below by $a_{0}$, $\lim_{n \to \infty} b_{n} = b$ exists.  Since
$a_{n} < b_{n}$ for all $n$, we also have $a \le b$.  In the end, if $a \in S$, player A
wins; otherwise, if $a \notin S$, player B wins.  Figure~\ref{F:play} illustrates a play
of the Cantor game.

Baker uses the Cantor game to prove that the closed interval $[0, 1]$ is uncountable
\cite[p.377]{mB07} and that every perfect set is uncountable \cite[p.378]{mB07}.  He also
shows that player B has a winning strategy when the target set $S$ is countable
\cite[p.377]{mB07}.  After proving these results he poses three challenging questions
\cite[p.379]{mB07}:

\emph{Do there exist uncountable subsets of $[0, 1]$ for which:
\begin{enumerate}
\item Alice does not have a winning strategy;
\item Bob has a winning strategy;
\item neither Alice nor Bob has a winning strategy?
\end{enumerate}}

We will show that the answers to these questions depend upon which axioms of set theory we
assume.  First we assume only the Zermelo-Fraenkel axioms of set theory (ZF) and prove some
results about winning strategies for the Cantor game.  Then assuming the Axiom of
Determinacy in addition to ZF (ZF + AD), we shall prove that the answer to all three
questions is ``no.''  Finally, assuming the Axiom of Choice in addition to ZF (ZFC), we
will show that the answer to questions 1 and 3 is ``yes,'' and we will give evidence that
the answer to question 2 is likely to be ``no.''

\begin{figure}
\centering\includegraphics{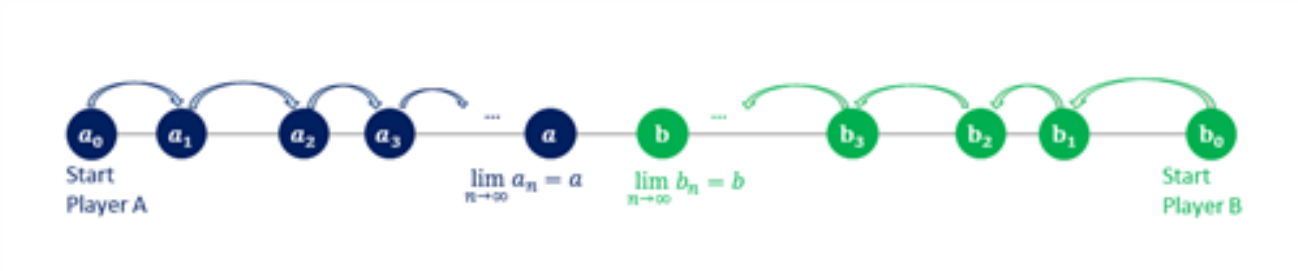}
\caption{A play of the Cantor game. Player A wins if $a \in S$.}\label{F:play}
\end{figure}

\section{Strategic definitions}\label{S:defs}
Thomas Jech begins his monograph on axiomatic set theory by discussing the eight
Zermelo-Fraenkel (ZF) axioms.  He then develops some consequences of those axioms and shows
how the real numbers can be constructed without additional assumptions \cite{tJ03}.  This
allows us to see which standard results of elementary real analysis follow from the ZF
axioms.  In this section and the following two sections we assume only the eight ZF axioms.
Once we have proved several fascinating results about the Cantor game, we shall assume
additional axioms and provide answers to our three questions.

Because our answers will require precise definitions and properties of strategies and
generalized Cantor sets, we begin by defining these fundamental concepts.  Our first
definition was inspired by Oxtoby's work on the related Banach-Mazur game in \cite[p.27]
{jO80}.  In the following definitions we consider the Cantor game on $[a_{0}, b_{0}]$,
where $a_{0}$ and $b_{0}$ are fixed real numbers satisfying $a_{0} < b_{0}$, with an
uncountable target set $S$.

\begin{definition}\label{D:strat}
A \emph{strategy} for player A is an infinite sequence of real-valued functions
$(f_{n})_{n \ge 0}$.  The function $f_{0}$ is required to have the domain $\{(a_{0},
b_{0})\}$, and its value must satisfy $a_{0} < f_{0}(a_{0}, b_{0}) < b_{0}$.  For each
$n \ge 2$ the function $f_{n-1}$ has domain $\{(a_{0}, b_{0}, a_{1}, b_{1}, \dots, a_{n-1},
b_{n-1}) \mid a_{0} < a_{1} < \dots < a_{n-1} < b_{n-1} < b_{n-2} < \dots < b_{0}\}$.  If
$a_{n} = f_{n-1}(a_{0}, b_{0}, \dots, a_{n-1}, b_{n-1})$, then we require that $a_{n-1} <
a_{n} < b_{n-1}$.

Likewise, a \emph{strategy} for player B is a sequence of real-valued functions
$(g_{n})_{n \ge 0}$.  The function $g_{0}$ is required to have the domain $\{(a_{0},
b_{0}, a_{1}) \mid a_{0} < a_{1} < b_{0}\}$, and its value must satisfy $a_{1} <
g_{0}(a_{0}, b_{0}, a_{1}) < b_{0}$.  For each $n \ge 2$ the function $g_{n-1}$ has domain
$\{(a_{0}, b_{0}, a_{1}, b_{1}, \dots, a_{n-1}, b_{n-1}, a_{n}) \mid a_{0} < a_{1} < \dots
< a_{n} < b_{n-1} < b_{n-2} < \dots < b_{0}\}$.  Furthermore, if $b_{n} = g_{n-1}(a_{0},
b_{0}, \dots, a_{n-1}, b_{n-1}, a_{n})$, then we require that $a_{n} < b_{n} < b_{n-1}$.
\end{definition}

\begin{definition}\label{D:play}
A \emph{play} of the game is an ordered pair of sequences $((a_{n})_{n \ge 0},\\ (b_{n})_{n
\ge 0})$, where $a_{n-1} < a_{n} < b_{n} < b_{n-1}$ for each $n \ge 1$.  The play is
\emph{consistent with} a given strategy for player A $(f_{n})_{n \ge 0}$ if and only if
$a_{n} = f_{n-1}(a_{0}, b_{0}, a_{1}, b_{1}, \dots, a_{n-1}, b_{n-1})$ for all $n \ge 1$.
Similarly, the play is \emph{consistent with} the strategy for player B $(g_{n})_{n \ge 0}$
if and only if $b_{n} = g_{n-1}(a_{0}, b_{0},\\ \dots, a_{n-1}, b_{n-1}, a_{n})$ for all
$n \ge 1$.
\end{definition}

\begin{definition}\label{D:limset}
The \emph{limit set} of a strategy $(f_{n})_{n \ge 0}$ for player A is defined as follows:
\[
   L((f_{n})) = \{\lim_{n \to \infty} a_{n} \mid ((a_{n}), (b_{n})) \text{
                is a play of the game consistent with }(f_{n})\}.
\]
Likewise, the \emph{limit set} of a strategy $(g_{n})_{n \ge 0}$ for player B is defined as
follows:
\[
   L((g_{n})) = \{\lim_{n \to \infty} a_{n} \mid ((a_{n}), (b_{n})) \text{
                is a play of the game consistent with }(g_{n})\}.
\]
We define a strategy $(f_{n})_{n \ge 0}$ for player A to be a \emph{winning strategy} if
and only if $L((f_{n})) \subset S$.  A strategy $(g_{n})_{n \ge 0}$ for player B is said to
be a \emph{winning strategy} if and only if $L((g_{n}))\subset [a_{0}, b_{0}] - S$.
\end{definition}

\section{Generalized Cantor sets}\label{S:gcs}
We define the concept of a generalized Cantor set as follows.

\begin{definition}\label{D:gc}
Let $I = [c, d]$, where $c < d$, and suppose that $(e_{n})_{n \ge 1}$ is a strictly
decreasing sequence of positive real numbers that converges to 0.  For each $n \ge 1$ and
each of the $2^{n}$ finite sequences of 0s and 1s, $(i_{1}, i_{2}, \dots, i_{n})$, let
$I_{i_{1}, i_{2}, \dots, i_{n}} = [c_{i_{1}, i_{2}, \dots, i_{n}}, d_{i_{1}, i_{2}, \dots,
i_{n}}]$ be a closed interval.  Suppose that the following properties hold:
\begin{enumerate}
\item $0 < d_{i_{1}, \dots, i_{n}} - c_{i_{1}, \dots, i_{n}} < e_{n}$ for all $n \ge 1$ and
      all $i_{1}, \dots, i_{n} \in \{0, 1\}$;
\item $I_{0}, I_{1} \subset I$ and $I_{i_{1}, \dots, i_{n}} \subset I_{i_{1}, \dots,
      i_{n-1}}$ for all $n \ge 2$ and all $i_{1}, \dots, i_{n} \in \{0, 1\}$;
\item $I_{0} \cap I_{1} = \emptyset$ and $I_{i_{1}, \dots, i_{n-1}, 0} \cap I_{i_{1},
      \dots, i_{n-1}, 1} = \emptyset$ for all $n \ge 2$ and all $i_{1}, \dots, i_{n-1} \in
      \{0, 1\}$.
\end{enumerate}

For each $n\ge 1$ let
\[
   C_{n} = \bigcup_{i_{1}, \dots, i_{n} \in \{0, 1\}} I_{i_{1}, i_{2}, \dots, i_{n}},
\]
and define
\[
   C = \bigcap_{n \ge 1} C_{n}.
\]
Then $C$ is called the \emph{generalized Cantor set} generated by the collection of
intervals
\[
   \mathcal{C} = \{I_{i_{1}, \dots, i_{n}} \mid n \ge 1\text{ and }i_{1}, \dots, i_{n} \in
                 \{0, 1\}\}.
\]
\end{definition}

Generalized Cantor sets have many of the familiar properties of the Cantor middle-thirds
set.  Before stating our main results on generalized Cantor sets, we first prove a lemma.

\begin{lemma}\label{L:gc}
Let $C$ be the generalized Cantor set generated by the collection of intervals
$\mathcal{C}$ in Definition~\ref{D:gc}.  Then for any $n \ge 1$, if $I_{i_{1}, \dots,
i_{n}} \cap I_{j_{1}, \dots, j_{n}} \neq \emptyset$, then $i_{1} = j_{1}, \dots, i_{n} =
j_{n}$.
\end{lemma}

\begin{proof}
When $n = 1$, if $i_{1} \ne j_{1}$, then $I_{i_{1}} \cap I_{j_{1}} = \emptyset$, and so the
result holds.  Now assume that the result holds for $n - 1$.  Suppose that $x \in I_{i_{1},
\dots, i_{n}} \cap I_{j_{1}, \dots, j_{n}}$.  By (2) of Definition~\ref{D:gc} we see that
$x \in I_{i_{1}, \dots, i_{n-1}} \cap I_{j_{1}, \dots, j_{n-1}}$.  Now, by the induction
hypothesis, $i_{1} = j_{1}, \dots, i_{n-1} = j_{n-1}$.  Hence $x \in I_{i_{1}, \dots,
i_{n-1}, i_{n}} \cap I_{i_{1}, \dots, i_{n-1}, j_{n}}$.  By (3) of Definition~\ref{D:gc},
these intervals are disjoint if $i_{n} \ne j_{n}$.  It follows that $i_{n} = j_{n}$, and so
the result is true by induction.
\end{proof}

\begin{theorem}\label{T:gc}
Let $C$ be the generalized Cantor set generated by the collection of intervals
$\mathcal{C}$ in Definition~\ref{D:gc}.  Then the following properties hold:
\begin{enumerate}
\item If $(i_{n})_{n \ge 1}$ is any sequence in $\{0, 1\}$, then $\lim_{n \to \infty}
      c_{i_{1}, \dots, i_{n}} =\\ \lim_{n \to \infty} d_{i_{1}, \dots, i_{n}}$;
\item $x \in C$ if and only if there is a unique sequence $(i_{n})_{n \ge 1}$ in $\{0, 1\}$
      such that $\{x\} = \bigcap_{n \ge 1} I_{i_{1}, \dots, i_{n}}$.  Furthermore, since
      $c_{i_{1}, \dots, i_{n}} \le x \le d_{i_{1}, \dots, i_{n}}$ for all $n$ and the two
      sequences have equal limits, we must have $\lim_{n \to \infty} c_{i_{1}, \dots,
      i_{n}} = x = \lim_{n \to \infty} d_{i_{1}, \dots, i_{n}}$;
\item $C$ is a perfect set.
\end{enumerate}
\end{theorem}

\begin{proof}
\begin{enumerate}
\item
Let $(i_{n})_{n \ge 1}$ be any sequence in $\{0, 1\}$.  By (2) of Definition~\ref{D:gc},
$I_{i_{1}, \dots, i_{n}} \subset I_{i_{1}, \dots, i_{n-1}}$, so that $c_{i_{1}, \dots,
i_{n-1}} \le c_{i_{1}, \dots, i_{n}}$.  Hence $(c_{i_{1}, \dots, i_{n}})_{n \ge 1}$ is a
monotonically increasing sequence bounded above by $d$, which implies that $\lim_{n \to
\infty} c_{i_{1}, \dots, i_{n}}$ exists.  A similar argument shows that $\lim_{n \to
\infty} d_{i_{1}, \dots, i_{n}}$ exists.  Since $\lim_{n \to \infty} e_{n} = 0$ and $0 <
d_{i_{1}, \dots, i_{n}} - c_{i_{1}, \dots, i_{n}} < e_{n}$ by (1) of Definition~\ref{D:gc},
we have $\lim_{n \to \infty}(d_{i_{1}, \dots, i_{n}} - c_{i_{1}, \dots, i_{n}}) = 0$, and
so $\lim_{n \to \infty} c_{i_{1}, \dots, i_{n}} = \lim_{n \to \infty} d_{i_{1}, \dots,
i_{n}}$.
\item
Suppose that $x \in C$; then $x \in C_{n}$ for all $n \ge 1$.  We proceed by induction on
$n$.  When $n = 1$, we know that $x \in C_{1} = I_{0} \cup I_{1}$, and we choose $i_{1}
\in \{0, 1\}$ such that $x \in I_{i_{1}}$.  Now assume that $i_{1}, i_{2}, \dots, i_{n-1}$
have been chosen such that $x \in I_{i_{1}}, x \in I_{i_{1}, i_{2}}, \dots, x \in I_{i_{1},
\dots, i_{n-1}}$.  Since $x \in C_{n}$, there exist $j_{1}, j_{2}, \dots, j_{n-1}, i_{n}
\in \{0, 1\}$ such that $x \in I_{j_{1}, \dots, j_{n-1}, i_{n}}$.  By (2) of Definition~
\ref{D:gc}, $x \in I_{j_{1}, \dots, j_{n-1}}$.  By Lemma~\ref{L:gc} this implies that
$i_{1} = j_{1}, \dots, i_{n-1} = j_{n-1}$.  Thus $x \in I_{i_{1}, \dots, i_{n}}$.  By
induction, it follows that we have defined a sequence $(i_{n})_{n \ge 1}$ such that $x \in
\bigcap_{n \ge 1} I_{i_{1}, \dots, i_{n}}$.

Next, assume that $y \in \bigcap_{n \ge 1} I_{i_{1}, \dots, i_{n}}$.  Then $x, y \in
[c_{i_{1}, \dots, i_{n}}, d_{i_{1}, \dots, i_{n}}]$ for $n \ge 1$, and so $|y - x| \le
d_{i_{1}, \dots, i_{n}} - c_{i_{1}, \dots, i_{n}} < e_{n}$ for all $n$.  Since $(e_{n})_{n
\ge 1}$ converges to $0$, it follows that $|y - x| = 0$, and so $y = x$.  Thus there exists
a sequence $(i_{n})_{n \ge 1}$ such that $\{x\} = \bigcap_{n \ge 1} I_{i_{1}, \dots,
i_{n}}$.  Now assume that there is another sequence $(j_{n})_{n \ge 1}$ having the same
property, and let $n \ge 1$ be given.  Then $x \in I_{i_{1}, \dots, i_{n}} \cap I_{j_{1},
\dots, j_{n}}$, and so by Lemma~\ref{L:gc}, we have $i_{n} = j_{n}$.  Since $n \ge 1$ was
arbitrary, it follows that the two sequences are identical.  This shows that the sequence
$(i_{n})_{n \ge 1}$ is uniquely determined.

Conversely, assume that there is a sequence $(i_{n})_{n \ge 1}$ such that $\{x\} =
\bigcap_{n \ge 1} I_{i_{1}, \dots, i_{n}}$.  Since $x \in I_{i_{1}, \dots, i_{n}} \subset
C_{n}$ for each $n \ge 1$, it follows that $x \in C$, thereby completing the proof.
\item
Since $I_{i_{1}, \dots, i_{n}}$ is closed, each $C_{n}$ is a finite union of closed sets,
so that each $C_{n}$ is closed.  So $C$ is an intersection of closed sets, and it follows
that $C$ is closed.  Now assume that $x \in C$.  Then there is a sequence $(i_{n})_{n \ge
1}$ such that $\{x\} = \bigcap_{n \ge 1} I_{i_{1}, \dots, i_{n}}$.  Let $\epsilon > 0$ be
given.  Then there exists $N$ such that $e_{N} < \epsilon$.  Define a sequence $(j_{n})_{n
\ge 1}$ as follows:
\[
   j_{n} =
   \begin{cases}
   1 - i_{n} &\text{if $n = N + 1$;}\\
   i_{n} &\text{if $n \neq N + 1$.}
   \end{cases}
\]
Let $y$ be such that $\{y\} = \bigcap_{n \ge 1} I_{j_{1}, \dots, j_{n}}$.  Then $y \in C$,
and by part (2), $y \neq x$.  Observe that $x, y \in I_{i_{1}, \dots, i_{N}}$; hence
$|x - y| \le d_{i_{1}, \dots, i_{N}} - c_{i_{1}, \dots, i_{N}} < e_{N} < \epsilon$, so that
$y \in (x - \epsilon, x + \epsilon)$.  This shows that $x$ is a limit point of $C$.  Since
$x$ was arbitrary, it follows that $C$ is a perfect set.
\end{enumerate}
\end{proof}

\section{The Cantor Game and the Zermelo-Fraenkel Axioms}\label{S:ZF}
Using our precise definitions of strategies for players A and B and those properties of
generalized Cantor sets given by Theorem ~\ref{T:gc}, we can now prove some remarkably
powerful results about limit sets for both players.  Once again we note that we are
assuming only the axioms of ZF.

First we will show that the limit set of any strategy for player A contains a generalized
Cantor set.  Then we will use this result to characterize those target sets for which
player A has a winning strategy.  Next we will show that the limit set of any strategy for
player B also contains a generalized Cantor set, and we will use this result to provide
necessary conditions for player B to have a winning strategy.  Finally, we will introduce
the Perfect Set Property and show that it has a close relationship to the Cantor game.

\subsection{Results for player A}\label{SS:a}

\begin{theorem}\label{T:a}
In the Cantor game on $[a_{0}, b_{0}]$ with an uncountable target set $S$, the limit set of
any strategy for player A contains a generalized Cantor set.
\end{theorem}

\begin{proof}
Suppose that the uncountable target set $S$ for the Cantor game on $[a_{0}, b_{0}]$ is
given, and let $(f_{n})_{n \ge 0}$ be a strategy for player A.  Assume that we have
constructed a fixed enumeration $\mathbb{Q}_{0}$ of the rational numbers in
$[a_{0}, b_{0}]$.  Let $e_{n} = \frac{b_{0} - a_{0}}{2^{n}}$ for $n \ge 1$, and define
$a_{1} = f_{0}(a_{0}, b_{0})$, $c = a_{1}$, $d = b_{0}$, $u = \frac{c + d}{2}$, and $I =
[c, d]$.  This defines the interval and the bounding sequence that we will use for the
construction of a generalized Cantor set.

Let $d_{1}$ be the first element of $\mathbb{Q}_{0}$ (the one of least index) such that
$c < d_{1} < u$, and let
\[
   c_{1} = f_{1}(a_{0}, b_{0}, c, d_{1}).
\]
Define $d_{0}$ to be the first element of $\mathbb{Q}_{0}$ such that $c < d_{0} < c_{1}$,
and let
\[
   c_{0} = f_{1}(a_{0}, b_{0}, c, d_{0}).
\]
Also define
\[
   I_{0} = [c_{0}, d_{0}];
\]
\[
   I_{1} = [c_{1}, d_{1}];
\]
\[
   u_{0} = \frac{c_{0} + d_{0}}{2};
\]
and
\[
   u_{1} = \frac{c_{1} + d_{1}}{2}.
\]
Then we have
\[
   c < c_{0} < d_{0} < c_{1} < d_{1} < u < d.
\]
Hence
\[
   I_{0}, I_{1} \subset [c, u] \subset I
\]
and
\[
   I_{0} \cap I_{1} = \emptyset.
\]
By the preceding inequality we have
\[
   0 < d_{i_{1}} - c_{i_{1}} < u - c = \frac{d - c}{2} = \frac{b_{0} - a_{1}}{2} <
   \frac{b_{0} - a_{0}}{2} = e_{1}
\]
for $i_{1} \in \{0, 1\}$.  Observe that the left endpoints of $I_{0}$ and $I_{1}$ are
defined by using player A's strategy based upon two rational numbers that might be chosen
by player B.  %Figure~\ref{F:gc} illustrates the first step of this construction.

%\begin{figure}
%\centering\includegraphics{generalizedcantorset.pdf}
%\caption{The first step in the construction of the generalized Cantor set when player A
%has a given strategy.}\label{F:gc}
%\end{figure}

Next we define
\[
   d_{1, 1} = \text{first element of }\mathbb{Q}_{0}\text{ such that } c_{1} < d_{1, 1} <
              u_{1};
\]
\[
   c_{1, 1} = f_{2}(a_{0}, b_{0}, c, d_{1}, c_{1}, d_{1, 1});
\]
\[
   d_{1, 0} = \text{first element of }\mathbb{Q}_{0}\text{ such that } c_{1} < d_{1, 0} <
              c_{1, 1};
\]
\[
   c_{1, 0} = f_{2}(a_{0}, b_{0}, c, d_{1}, c_{1}, d_{1, 0});
\]
\[
   d_{0, 1} = \text{first element of }\mathbb{Q}_{0}\text{ such that } c_{0} < d_{0, 1} <
              u_{0};
\]
\[
   c_{0, 1} = f_{2}(a_{0}, b_{0}, c, d_{0}, c_{0}, d_{0, 1});
\]
\[
   d_{0, 0} = \text{first element of }\mathbb{Q}_{0}\text{ such that } c_{0} < d_{0, 0} <
              c_{0, 1};
\]
and
\[
   c_{0, 0} = f_{2}(a_{0}, b_{0}, c, d_{0}, c_{0}, d_{0, 0}).
\]
Let 
\[
   I_{i_{1}, i_{2}} = [c_{i_{1}, i_{2}}, d_{i_{1}, i_{2}}]
\]
and
\[
   u_{i_{1}, i_{2}} = \frac{c_{i_{1}, i_{2}} + d_{i_{1}, i_{2}}}{2}
\]
for $i_{1}, i_{2} \in \{0, 1\}$.  Observe once again that the left endpoint of each
interval $I_{i_{1}, i_{2}}$ is defined by using player A's strategy based upon a rational
number that might be chosen by player B.

Now assume that for each $k$ with $2 \le k \le n$, the collections of real numbers
$\{c_{i_{1}, \dots, i_{k}} \mid i_{1}, \dots, i_{k} \in \{0, 1\}\}$, $\{d_{i_{1},
\dots, i_{k}} \mid i_{1}, \dots, i_{k} \in \{0, 1\}\}$, and $\{u_{i_{1}, \dots, i_{k}} \mid
i_{1}, \dots, i_{k} \in \{0, 1\}\}$ and the collection of closed intervals $\{I_{i_{1},
\dots, i_{k}} \mid i_{1}, \dots, i_{k} \in \{0, 1\}\}$ have been defined and satisfy the
following properties for all choices of $i_{1}, i_{2}, \dots, i_{k} \in [0, 1]$:
\begin{enumerate}
\item $d_{i_{1}, \dots, i_{k-1}, 1} = \text{first element of }\mathbb{Q}_{0}\text{
      belonging to }(c_{i_{1}, \dots, i_{k-1}}, u_{i_{1}, \dots, i_{k-1}})$;
\item $c_{i_{1}, \dots, i_{k-1}, 1} = f_{k}(a_{0}, b_{0}, c, d_{i_{1}}, c_{i_{1}}, \dots,
      d_{i_{1}, \dots, i_{k-1}}, c_{i_{1}, \dots, i_{k-1}}, d_{i_{1}, \dots, i_{k-1}, 1})$;
\item $d_{i_{1}, \dots, i_{k-1}, 0} = \text{first element of }\mathbb{Q}_{0}\text{
      belonging to } (c_{i_{1}, \dots, i_{k-1}}, c_{i_{1}, \dots, i_{k-1}, 1})$;
\item $c_{i_{1}, \dots, i_{k-1}, 0} = f_{k}(a_{0}, b_{0}, c, d_{i_{1}}, c_{i_{1}}, \dots,
      d_{i_{1}, \dots, i_{k-1}}, c_{i_{1}, \dots, i_{k-1}}, d_{i_{1}, \dots, i_{k-1}, 0})$;
\item $I_{i_{1}, \dots, i_{k}} = [c_{i_{1}, \dots, i_{k}}, d_{i_{1}, \dots, i_{k}}]$; and
\item $u_{i_{1}, \dots, i_{k}} = \frac{c_{i_{1}, \dots, i_{k}} + d_{i_{1}, \dots, i_{k}}}
      {2}$.
\end{enumerate}
Then define collections of real numbers $\{c_{i_{1}, \dots, i_{n+1}} \mid i_{1}, \dots,
i_{n+1} \in \{0, 1\}\}$, $\{d_{i_{1}, \dots, i_{n+1}} \mid i_{1}, \dots, i_{n+1} \in
\{0, 1\}\}$, and $\{u_{i_{1}, \dots, i_{n+1}} \mid i_{1}, \dots, i_{n+1} \in \{0, 1\}\}$,
as well as a collection of closed intervals $\{I_{i_{1}, \dots, i_{n+1}} \mid i_{1}, \dots,
i_{n+1} \in \{0, 1\}\}$ with the following properties for each choice of $i_{1}, i_{2},
\dots, i_{n+1} \in [0, 1]$:
\begin{enumerate}
\item $d_{i_{1}, \dots, i_{n}, 1} = \text{first element of }\mathbb{Q}_{0}\text{ belonging
      to }(c_{i_{1}, \dots, i_{n}}, u_{i_{1}, \dots, i_{n}})$;
\item $c_{i_{1}, \dots, i_{n}, 1} = f_{n+1}(a_{0}, b_{0}, c, d_{i_{1}}, c_{i_{1}}, \dots,
      d_{i_{1}, \dots, i_{n}}, c_{i_{1}, \dots, i_{n}}, d_{i_{1}, \dots, i_{n}, 1})$;
\item $d_{i_{1}, \dots, i_{n}, 0} = \text{first element of }\mathbb{Q}_{0}\text{ belonging
      to } (c_{i_{1}, \dots, i_{n}}, c_{i_{1}, \dots, i_{n}, 1})$;
\item $c_{i_{1}, \dots, i_{n}, 0} = f_{n+1}(a_{0}, b_{0}, c, d_{i_{1}}, c_{i_{1}}, \dots,
      d_{i_{1}, \dots, i_{n}}, c_{i_{1}, \dots, i_{n}}, d_{i_{1}, \dots, i_{n}, 0})$;
\item $I_{i_{1}, \dots, i_{n+1}} = [c_{i_{1}, \dots, i_{n+1}}, d_{i_{1}, \dots, i_{n+1}}]$;
      and
\item $u_{i_{1}, \dots, i_{n+1}} = \frac{c_{i_{1}, \dots, i_{n+1}} + d_{i_{1}, \dots,
      i_{n+1}}}{2}$.
\end{enumerate}
By induction the first set of six properties holds for the collections of numbers and
intervals defined for every $k \ge 2$ and every choice of $i_{1}, i_{2}, \dots, i_{k} \in
[0, 1]$.

Now fix $n \ge 2$, and let $k = n$.  Applying properties (1) - (4) together with the
inequality for $f_{n}$ in Definition~\ref{D:play}, we see that for all choices of $i_{1},
i_{2}, \dots, i_{n-1} \in [0, 1]$,
\[
   c_{i_{1}, \dots, i_{n-1}} < c_{i_{1}, \dots, i_{n-1}, 0} < d_{i_{1}, \dots, i_{n-1},
   0} < c_{i_{1}, \dots, i_{n-1}, 1}
\]
and
\[
   c_{i_{1}, \dots, i_{n-1}, 1} < d_{i_{1}, \dots, i_{n-1}, 1} < u_{i_{1}, \dots,
   i_{n-1}} < d_{i_{1}, \dots, i_{n-1}}.
\]
Hence
\[
   I_{i_{1}, \dots, i_{n-1}, 0}, I_{i_{1}, \dots, i_{n-1}, 1} \subset [c_{i_{1}, \dots,
   i_{n-1}}, u_{i_{1}, \dots, i_{n-1}}] \subset I_{i_{1}, \dots, i_{n-1}}
\]
and
\[
   I_{i_{1}, \dots, i_{n-1}, 0} \cap I_{i_{1}, \dots, i_{n-1}, 1} = \emptyset.
\]
Combining these relations with those of the construction for $n = 1$, we see that
conditions (2) and (3) of Definition~\ref{D:gc} are satisfied.

To see that condition (1) of that definition also holds, we proceed by induction.  Recall
that we have defined $e_{n} = \frac{b_{0} - a_{0}}{2^{n}}$ for $n \ge 1$, and we have
already shown that $0 < d_{i_{1}} - c_{i_{1}} < e_{1}$ for $i_{1} \in \{0, 1\}$.  Now
assume that $0 < d_{i_{1}, \dots, i_{n-1}} - c_{i_{1}, \dots, i_{n-1}} < e_{n-1}$ for all
choices of $i_{1}, \dots, i_{n-1} \in \{0, 1\}$, and let $i_{1}, \dots, i_{n} \in \{0,
1\}$ be given.  Then $d_{i_{1}, \dots, i_{n}} - c_{i_{1}, \dots, i_{n}} > 0$ and
\[
   d_{i_{1}, \dots, i_{n}} - c_{i_{1}, \dots, i_{n}} < u_{i_{1}, \dots, i_{n-1}} -
   c_{i_{1}, \dots, i_{n-1}} = \frac{d_{i_{1}, \dots, i_{n-1}} - c_{i_{1}, \dots, i_{n-1}}}
   {2} < \frac{e_{n-1}}{2} = e_{n}.
\]
Hence condition (1) of Definition~\ref{D:gc} holds.

Now let $C$ be the generalized Cantor set generated by $\mathcal{C} = \{I_{i_{1}, \dots,
i_{n}} \mid n \ge 1\text{ and }i_{1}, \dots, i_{n} \in \{0, 1\}\}$.  We claim that
$C \subset L((f_{n})_{n \ge 0})$.  To prove this, let $x \in C$ be given.  By Theorem~
\ref{T:gc} there is a unique sequence $(i_{n})_{n \ge 1}$ such that $\{x\} = \bigcap_{n \ge
1} I_{i_{1}, \dots, i_{n}}$.  We define a play of the Cantor game as follows.  First note
that $a_{0}$ and $b_{0}$ are already given.  Now let $a_{1} = c$, $b_{1} = d_{i_{1}}$, and
define $a_{n} = c_{i_{1}, \dots, i_{n-1}}$ and $b_{n} = d_{i_{1}, \dots, i_{n}}$ for all
$n \ge 2$.  Then $a_{1} = f_{0}(a_{0}, b_{0})$, $a_{2} = c_{i_{1}} = f_{1}(a_{0}, b_{0}, c,
d_{i_{1}}) = f_{1}(a_{0}, b_{0}, a_{1}, b_{1})$, and for $n \ge 3$, $a_{n} = c_{i_{1},
\dots, i_{n-1}} = f_{n-1}(a_{0}, b_{0}, c, d_{i_{1}}, c_{i_{1}}, \dots, d_{i_{1}, \dots,
i_{n-2}}, c_{i_{1}, \dots, i_{n-2}}, d_{i_{1}, \dots, i_{n-1}}) = f_{n-1}(a_{0}, b_{0},
a_{1}, b_{1}, a_{2}, \dots, b_{n-2}, a_{n-1}, b_{n-1}) = f_{n-1}(a_{0}, b_{0}, a_{1},
b_{1}, \dots, a_{n-1},\\ b_{n-1})$.  Thus the play of the game $((a_{n})_{n \ge 0},
(b_{n})_{n \ge 0})$ is consistent with the given strategy $(f_{n})_{n \ge 0}$ for player A.
Now observe that $\lim_{n \to \infty} a_{n} = \lim_{n \to \infty} c_{i_{1}, \dots, i_{n-1}}
= x$, so that $x \in L((f_{n})_{n \ge 0})$.  It follows that the generalized Cantor set we
have defined is a subset of player A's limit set.
\end{proof}

As a consequence of Theorem~\ref{T:a} we obtain a characterization of those target sets $S$
for which player A has a winning strategy.

\begin{theorem}\label{T:mainA}
In the Cantor game on $[a_{0}, b_{0}]$ with an uncountable target set $S$, the following
three statements are equivalent:
\begin{enumerate}
\item Player A has a winning strategy;
\item $S$ contains a generalized Cantor set;
\item $S$ contains a perfect set.
\end{enumerate}
\end{theorem}

\begin{proof}
(1) $\Rightarrow$ (2)
If player A has a winning strategy, then $S$ contains the limit set of that strategy.  By
Theorem~\ref{T:a}, this limit set contains a generalized Cantor set.  Thus $S$ contains a
generalized Cantor set.

(2) $\Rightarrow$ (3)
This is true according to Theorem~\ref{T:gc} since every generalized Cantor set is a
perfect set.

(3) $\Rightarrow$ (1)
This was proved in \cite[p.379]{mB07}.
\end{proof}

\subsection{Results for player B}\label{SS:b}
We now prove that similar results hold for player B's strategy, although the implications
of these results will be less far-reaching than the results for player A.  The proof of
our next result is very similar to the proof of Theorem~\ref{T:a} for player A.  The main
difference is that the current argument's construction proceeds upward from the midpoint of
each subinterval, while the construction in Theorem~\ref{T:a} proceeds downward.

\begin{theorem}\label{T:b}
In the Cantor game on $[a_{0}, b_{0}]$ with an uncountable target set $S$, the limit set of
any strategy for player B contains a generalized Cantor set.
\end{theorem}

\begin{proof}
Suppose that the uncountable target set $S$ for the Cantor game on $[a_{0}, b_{0}]$ is
given, and let $(g_{n})_{n \ge 0}$ be a strategy for player B.  Assume that we have
constructed a fixed enumeration $\mathbb{Q}_{0}$ of the rational numbers in
$[a_{0}, b_{0}]$.  Let $e_{n} = \frac{b_{0} - a_{0}}{2^{n}}$ for $n \ge 1$, and define $c =
a_{0}$, $d = b_{0}$, $u = \frac{c + d}{2}$, and $I = [c, d]$.  This defines the interval
and the bounding sequence that we will use for the construction of a generalized Cantor
set.

Let $c_{0}$ be the first element of $\mathbb{Q}_{0}$ (the one of least index) such that
$u < c_{0} < d$, and let
\[
   d_{0} = g_{0}(a_{0}, b_{0}, c_{0}).
\]
Define $c_{1}$ to be the first element of $\mathbb{Q}_{0}$ such that $d_{0} < c_{1} < d$,
and let
\[
   d_{1} = g_{0}(a_{0}, b_{0}, c_{1}).
\]
Also define
\[
   I_{0} = [c_{0}, d_{0}];
\]
\[
   I_{1} = [c_{1}, d_{1}];
\]
\[
   u_{0} = \frac{c_{0} + d_{0}}{2};
\]
and
\[
   u_{1} = \frac{c_{1} + d_{1}}{2}.
\]
Then we have
\[
   c < u < c_{0} < d_{0} < c_{1} < d_{1} < d.
\]
Hence
\[
   I_{0}, I_{1} \subset [u, d] \subset I
\]
and
\[
   I_{0} \cap I_{1} = \emptyset.
\]
By the preceding inequality we have
\[
   0 < d_{i_{1}} - c_{i_{1}} < d - u = \frac{d - c}{2} = \frac{b_{0} - a_{0}}{2} = e_{1}
\]
for $i_{1} \in \{0, 1\}$.  Observe that the right endpoints of $I_{0}$ and $I_{1}$ are
defined by using player B's strategy based upon two rational numbers that might be chosen
by player A.

Next we define
\[
   c_{0, 0} = \text{first element of }\mathbb{Q}_{0}\text{ such that } u_{0} < c_{0, 0} <
              d_{0};
\]
\[
   d_{0, 0} = g_{1}(a_{0}, b_{0}, c_{0}, d_{0}, c_{0, 0});
\]
\[
   c_{0, 1} = \text{first element of }\mathbb{Q}_{0}\text{ such that } d_{0, 0} < c_{0, 1}
              < d_{0};
\]
\[
   d_{0, 1} = g_{1}(a_{0}, b_{0}, c_{0}, d_{0}, c_{0, 1});
\]
\[
   c_{1, 0} = \text{first element of }\mathbb{Q}_{0}\text{ such that } u_{1} < c_{1, 0}
              < d_{1};
\]
\[
   d_{1, 0} = g_{1}(a_{0}, b_{0}, c_{1}, d_{1}, c_{1, 0});
\]
\[
   c_{1, 1} = \text{first element of }\mathbb{Q}_{0}\text{ such that } d_{1, 0} < c_{1, 1}
              < d_{1};
\]
and
\[
   d_{1, 1} = g_{1}(a_{0}, b_{0}, c_{1}, d_{1}, c_{1, 1}).
\]
Let 
\[
   I_{i_{1}, i_{2}} = [c_{i_{1}, i_{2}}, d_{i_{1}, i_{2}}]
\]
and
\[
   u_{i_{1}, i_{2}} = \frac{c_{i_{1}, i_{2}} + d_{i_{1}, i_{2}}}{2}
\]
for $i_{1}, i_{2} \in \{0, 1\}$.  Observe once again that the right endpoint of each
interval $I_{i_{1}, i_{2}}$ is defined by using player B's strategy based upon a rational
number that might be chosen by player A.

Now assume that for each $k$ with $2 \le k \le n$, the collections of real numbers
$\{c_{i_{1}, \dots, i_{k}} \mid i_{1}, \dots, i_{k} \in \{0, 1\}\}$, $\{d_{i_{1},
\dots, i_{k}} \mid i_{1}, \dots, i_{k} \in \{0, 1\}\}$, and $\{u_{i_{1}, \dots, i_{k}} \mid
i_{1}, \dots, i_{k} \in \{0, 1\}\}$ and the collection of closed intervals $\{I_{i_{1},
\dots, i_{k}} \mid i_{1}, \dots, i_{k} \in \{0, 1\}\}$ have been defined and satisfy the
following properties for all choices of $i_{1}, i_{2}, \dots, i_{k} \in [0, 1]$:
\begin{enumerate}
\item $c_{i_{1}, \dots, i_{k-1}, 0} = \text{first element of }\mathbb{Q}_{0}\text{
      belonging to }(u_{i_{1}, \dots, i_{k-1}}, d_{i_{1}, \dots, i_{k-1}})$;
\item $d_{i_{1}, \dots, i_{k-1}, 0} = g_{k-1}(a_{0}, b_{0}, c_{i_{1}}, d_{i_{1}}, \dots,
      c_{i_{1}, \dots, i_{k-1}}, d_{i_{1}, \dots, i_{k-1}}, c_{i_{1}, \dots, i_{k-1}, 0})$;
\item $c_{i_{1}, \dots, i_{k-1}, 1} = \text{first element of }\mathbb{Q}_{0}\text{
      belonging to } (d_{i_{1}, \dots, i_{k-1}, 0}, d_{i_{1}, \dots, i_{k-1}})$;
\item $d_{i_{1}, \dots, i_{k-1}, 1} = g_{k-1}(a_{0}, b_{0}, c_{i_{1}}, d_{i_{1}}, \dots,
      c_{i_{1}, \dots, i_{k-1}}, d_{i_{1}, \dots, i_{k-1}}, c_{i_{1}, \dots, i_{k-1}, 1})$;
\item $I_{i_{1}, \dots, i_{k}} = [c_{i_{1}, \dots, i_{k}}, d_{i_{1}, \dots, i_{k}}]$; and
\item $u_{i_{1}, \dots, i_{k}} = \frac{c_{i_{1}, \dots, i_{k}} + d_{i_{1}, \dots, i_{k}}}
      {2}$.
\end{enumerate}
Then define collections of real numbers $\{c_{i_{1}, \dots, i_{n+1}} \mid i_{1}, \dots,
i_{n+1} \in \{0, 1\}\}$, $\{d_{i_{1}, \dots, i_{n+1}} \mid i_{1}, \dots, i_{n+1} \in
\{0, 1\}\}$, and $\{u_{i_{1}, \dots, i_{n+1}} \mid i_{1}, \dots, i_{n+1} \in \{0, 1\}\}$,
as well as a collection of closed intervals $\{I_{i_{1}, \dots, i_{n+1}} \mid i_{1}, \dots,
i_{n+1} \in \{0, 1\}\}$ with the following properties for each choice of $i_{1}, i_{2},
\dots, i_{n+1} \in [0, 1]$:
\begin{enumerate}
\item $c_{i_{1}, \dots, i_{n}, 0} = \text{first element of }\mathbb{Q}_{0}\text{ belonging
      to }(u_{i_{1}, \dots, i_{n}}, d_{i_{1}, \dots, i_{n}})$;
\item $d_{i_{1}, \dots, i_{n}, 0} = g_{n}(a_{0}, b_{0}, c_{i_{1}}, d_{i_{1}}, \dots,
      c_{i_{1}, \dots, i_{n}}, d_{i_{1}, \dots, i_{n}}, c_{i_{1}, \dots, i_{n}, 0})$;
\item $c_{i_{1}, \dots, i_{n}, 1} = \text{first element of }\mathbb{Q}_{0}\text{ belonging
      to } (d_{i_{1}, \dots, i_{n}, 0}, d_{i_{1}, \dots, i_{n}})$;
\item $d_{i_{1}, \dots, i_{n}, 1} = g_{n}(a_{0}, b_{0}, c_{i_{1}}, d_{i_{1}}, \dots,
      c_{i_{1}, \dots, i_{n}}, d_{i_{1}, \dots, i_{n}}, c_{i_{1}, \dots, i_{n}, 1})$;
\item $I_{i_{1}, \dots, i_{n+1}} = [c_{i_{1}, \dots, i_{n+1}}, d_{i_{1}, \dots, i_{n+1}}]$;
and
\item $u_{i_{1}, \dots, i_{n+1}} = \frac{c_{i_{1}, \dots, i_{n+1}} + d_{i_{1}, \dots,
      i_{n+1}}}{2}$.
\end{enumerate}
By induction the first set of six properties holds for the collections of numbers and
intervals defined for every $k \ge 2$ and every choice of $i_{1}, i_{2}, \dots, i_{k} \in
[0, 1]$.

Now fix $n \ge 2$, and let $k = n$.  Applying properties (1) - (4) together with the
inequality for $g_{n}$ in Definition~\ref{D:play}, we see that for all choices of $i_{1},
i_{2}, \dots, i_{n-1} \in [0, 1]$,
\[
   c_{i_{1}, \dots, i_{n-1}} < u_{i_{1}, \dots, i_{n-1}} < c_{i_{1}, \dots, i_{n-1}, 0} <
   d_{i_{1}, \dots, i_{n-1}, 0}
\]
and
\[
   d_{i_{1}, \dots, i_{n-1}, 0} < c_{i_{1}, \dots, i_{n-1}, 1} < d_{i_{1}, \dots, i_{n-1},
   1} < d_{i_{1}, \dots, i_{n-1}}.
\]
Hence
\[
   I_{i_{1}, \dots, i_{n-1}, 0}, I_{i_{1}, \dots, i_{n-1}, 1} \subset [u_{i_{1}, \dots,
   i_{n-1}}, d_{i_{1}, \dots, i_{n-1}}] \subset I_{i_{1}, \dots, i_{n-1}}
\]
and
\[
   I_{i_{1}, \dots, i_{n-1}, 0} \cap I_{i_{1}, \dots, i_{n-1}, 1} = \emptyset.
\]
Combining these relations with those of the construction for $n = 1$, we see that
conditions (2) and (3) of Definition~\ref{D:gc} are satisfied.

To see that condition (1) of that definition also holds, we proceed by induction.  Recall
that we have defined $e_{n} = \frac{b_{0} - a_{0}}{2^{n}}$ for $n \ge 1$, and we have
already shown that $0 < d_{i_{1}} - c_{i_{1}} < e_{1}$ for $i_{1} \in \{0, 1\}$.  Now
assume that $0 < d_{i_{1}, \dots, i_{n-1}} - c_{i_{1}, \dots, i_{n-1}} < e_{n-1}$ for all
choices of $i_{1}, \dots, i_{n-1} \in \{0, 1\}$, and let $i_{1}, \dots, i_{n} \in \{0,
1\}$ be given.  Then $d_{i_{1}, \dots, i_{n}} - c_{i_{1}, \dots, i_{n}} > 0$ and
\[
   d_{i_{1}, \dots, i_{n}} - c_{i_{1}, \dots, i_{n}} < d_{i_{1}, \dots, i_{n-1}} -
   u_{i_{1}, \dots, i_{n-1}} = \frac{d_{i_{1}, \dots, i_{n-1}} - c_{i_{1}, \dots, i_{n-1}}}
   {2} < \frac{e_{n-1}}{2} = e_{n}.
\]
Hence condition (1) of Definition~\ref{D:gc} holds.

Now let $D$ be the generalized Cantor set generated by $\mathcal{D} = \{I_{i_{1}, \dots,
i_{n}} \mid n \ge 1\text{ and }i_{1}, \dots, i_{n} \in \{0, 1\}\}$.  We claim that
$D \subset L((g_{n})_{n \ge 0})$.  To prove this, let $x \in D$ be given.  By Theorem~
\ref{T:gc} there is a unique sequence $(i_{n})_{n \ge 1}$ such that $\{x\} = \bigcap_{n \ge
1} I_{i_{1}, \dots, i_{n}}$.  We define a play of the Cantor game as follows.  First note
that $a_{0}$ and $b_{0}$ are fixed.  Now let $a_{n} = c_{i_{1}, \dots, i_{n}}$ and $b_{n} =
d_{i_{1}, \dots, i_{n}}$ for all $n \ge 1$.  Then $b_{1} = d_{i_{1}} = g_{0}(c, d,
c_{i_{1}}) = g_{0}(a_{0}, b_{0}, a_{1})$, and for $n \ge 2$, $b_{n} = d_{i_{1}, \dots,
i_{n}} = g_{n-1}(c, d, c_{i_{1}}, d_{i_{1}}, \dots, c_{i_{1}, \dots, i_{n-1}}, d_{i_{1},
\dots, i_{n-1}}, c_{i_{1}, \dots, i_{n}}) = g_{n-1}(a_{0}, b_{0}, a_{1}, b_{1}, \dots,
a_{n-1}, b_{n-1}, a_{n})$.  Thus the play of the game $((a_{n})_{n \ge 0},\\ (b_{n})_{n
\ge 0})$ is consistent with the given strategy $(g_{n})_{n \ge 0}$ for player B.  Because
$\lim_{n \to \infty} a_{n} = \lim_{n \to \infty} c_{i_{1}, \dots, i_{n}} = x$, it follows
that $x \in L((g_{n})_{n \ge 0})$.  Thus player B's limit set contains the generalized
Cantor set $D$.
\end{proof}

The following lemma will be useful in several contexts.

\begin{lemma}\label{L:bern3}
Let $S$ be an arbitrary subset of the closed interval $[a, b]$, where $a < b$.  If $S$ does
not contain a perfect set, then $[a, b] - S$ is dense in $[a, b]$.
\end{lemma}

\begin{proof}
Suppose that $S$ does not contain a perfect set.  Let $x \in [a, b]$ be given, and suppose,
to reach a contradiction, that $x \notin \overline{[a, b] - S}$.  Assume first that $a < x
< b$.  Then there exists $\epsilon > 0$ such that $(x - \epsilon, x + \epsilon) \subset
(a, b)$ and $(x - \epsilon, x + \epsilon) \cap ([a, b] - S) = \emptyset$.  Hence
$(x - \epsilon, x + \epsilon) \subset S$.  This implies that $S$ contains the perfect set
$[x - \frac{\epsilon}{2}, x + \frac{\epsilon}{2}]$, contrary to our assumption.  Now assume
that $x = a$.  Then there exists $\epsilon > 0$ such that $[a, a + \epsilon) \cap ([a, b] -
S) = \emptyset$.  This implies that $S$ contains the perfect set $[a + \frac{\epsilon}{4},
a + \frac{3\epsilon}{4})]$, contrary to our assumption.  The case in which $x = b$ is
handled similarly.  It follows that $x \in \overline{[a, b] - S}$.  This proves that
$\overline{[a, b] - S} = [a, b]$.
\end{proof}

As a consequence of Theorem~\ref{T:b} and Lemma~\ref{L:bern3} we obtain necessary
conditions for player B to have a winning strategy.

\begin{theorem}\label{T:mainB}
In the Cantor game on $[a_{0}, b_{0}]$ with an uncountable target set $S$, if player B has
a winning strategy, then:
\begin{enumerate}
\item $[a_{0}, b_{0}] - S$ contains a generalized Cantor set;
\item $[a_{0}, b_{0}] - S$ contains a perfect set;
\item $[a_{0}, b_{0}] - S$ is dense in $[a_{0}, b_{0}]$.
\end{enumerate}
\end{theorem}

\begin{proof}
(1)
If player B has a winning strategy, then the complement of $S$ contains the limit set of
that strategy.  By Theorem~\ref{T:b}, this limit set contains a generalized Cantor set.
Thus the complement of $S$ contains a generalized Cantor set.

(2)
This is true according to Theorem~\ref{T:gc} since every generalized Cantor set is a
perfect set.

(3)
Since player B has a winning strategy, player A cannot have a winning strategy.  Then by
Theorem~\ref{T:mainA}, $S$ does not contain a perfect set.  It now follows from Lemma~
\ref{L:bern3} that $[a_{0}, b_{0}] - S$ is dense in $[a_{0}, b_{0}]$.
\end{proof}

\subsection{The Perfect Set Property}\label{SS:perf}
We now introduce the Perfect Set Property and show that it has a close relationship to the
Cantor game.

\begin{definition}\label{D:perf}
Let $S \subset \mathbb{R}$.  We say that $S$ has the \emph{perfect set property} if and
only if $S$ is countable or there exists a nonempty perfect set $P \subset S$.
Equivalently, if $S$ is uncountable, then $S$ must contain a nonempty perfect set.
\end{definition}

Up to this point we have been assuming only the eight axioms of Zermelo-Fraenkel set theory
(ZF); however, in our next theorem we will need to know that a countable union of countable
sets is countable.  Since this result is not known to be provable in ZF, we will need to
assume an additional axiom of set theory in order to use it.  One such axiom is the Axiom
of Countable Choice, AC$_{\omega}$, which is called the Countable Axiom of Choice by some
authors.  A precise statement of this axiom and some of its consequences can be found in
Jech's monograph \cite[Chapter 5]{tJ03}.

\begin{theorem}\label{T:perf1}
Assuming the Axiom of Countable Choice in addition to ZF (ZF + AC$_{\omega}$), the
following two statements are equivalent:
\begin{enumerate}
\item Every subset of $\mathbb{R}$ has the perfect set property.
\item For every $a_{0}, b_{0} \in \mathbb{R}$ with $a_{0} < b_{0}$ and for every
uncountable set $S \subset [a_{0}, b_{0}]$, player A has a winning strategy in the Cantor
game on $[a_{0}, b_{0}]$ with target set $S$. (That is, player A always has a winning
strategy.)
\end{enumerate} 
\end{theorem}

\begin{proof}
(1) $\Rightarrow$ (2)
Suppose that every subset of $\mathbb{R}$ has the perfect set property.  Let $a_{0}, b_{0}
\in \mathbb{R}$ with $a_{0} < b_{0}$, and let $S$ be any uncountable subset of $[a_{0},
b_{0}]$.  Then $S$ contains a nonempty perfect set, and so by Theorem~\ref{T:mainA}, player
A has a winning strategy in the Cantor game on $[a_{0}, b_{0}]$ with target set $S$.

(2) $\Rightarrow$ (1)
Now assume that condition (2) holds.  Let $S \subset \mathbb{R}$, and assume that $S$ is
uncountable.  If $S \cap [n, n + 1]$ were countable for every $n \in Z$, then $S = 
\cup_{n \in \mathbb{Z}} S \cap [n, n + 1]$ would be countable.  Hence $S \cap [n, n + 1]$
must be uncountable for some $n \in \mathbb{Z}$.  By condition (2), in the Cantor game on
$[n, n + 1]$ with target set $S \cap [n, n + 1]$, player A has a winning strategy.  Hence,
by Theorem~\ref{T:mainA}, $S \cap [n, n + 1]$ contains a nonempty perfect set, so that $S$
also contains a perfect set.  Thus every subset of $\mathbb{R}$ has the perfect set
property.
\end{proof}

\section{The Cantor Game and the Axiom of Determinacy}\label{S:AD}
Prior to Section~\ref{SS:perf} we assumed only the eight axioms of Zermelo-Fraenkel set
theory (ZF), and in that section we added the assumption of the Axiom of Countable Choice
(AC$_{\omega}$).  In order to provide definitive answers to Baker's three questions, we
must assume axioms in addition to those of ZF + AC$_{\omega}$.  As we will see, the answers
to these questions depend on which additional axioms we assume.

In this brief section we will assume that the Axiom of Determinacy (AD) holds in addition
to the Zermelo-Fraenkel axioms, and thus the results will be valid in the ZF + AD system.
In the following section we will assume that the Axiom of Choice holds in addition to the
Zermelo-Fraenkel axioms, and the results in that section will be valid in the ZFC system.
Since the Axiom of Countable Choice (AC$_{\omega}$) is a consequence of the Axiom of
Determinacy and of the Axiom of Choice, all of our previous results will continue to hold
in this section and the following one.  In Chapter 5 and Chapter 33 of his treatise
\cite{tJ03}, Jech offers precise statements of all of these axioms and many of their
consequences.  The most significant consequence for our purposes is the following.

\begin{theorem}\label{T:perf2}
Assuming the Axiom of Determinacy in addition to ZF (ZF + AD), every subset of $\mathbb{R}$
has the perfect set property.
\end{theorem}

Combining this result with Theorem~\ref{T:perf1}, we obtain the following result.

\begin{theorem}\label{T:perf3}
Assuming the Axiom of Determinacy in addition to ZF (ZF + AD), for every $a_{0}, b_{0} \in
\mathbb{R}$ with $a_{0} < b_{0}$ and for every uncountable set $S \subset [a_{0}, b_{0}]$,
player A has a winning strategy in the Cantor game on $[a_{0}, b_{0}]$ with target set $S$.
\end{theorem}

\begin{proof}
By Theorem~\ref{T:perf2}, (1) of Theorem~\ref{T:perf1} holds, and so (2) of that theorem
holds.
\end{proof}

As a corollary to Theorem~\ref{T:perf3}, we find that in ZF + AD, the answer to all three
questions is ``no.''

\begin{corollary}\label{C:perf}
Assuming the Axiom of Determinacy in addition to the Zermelo-Fraenkel axioms (ZF + AD), let
$a_{0}, b_{0} \in \mathbb{R}$ with $a_{0} < b_{0}$, and consider the Cantor game on
$[a_{0}, b_{0}]$.  Then
\begin{enumerate}
\item There does not exist an uncountable set $S \subset [a_{0}, b_{0}]$ such that player A
      does not have a winning strategy;
\item There does not exist an uncountable set $S \subset [a_{0}, b_{0}]$ such that player B
      has a winning strategy;
\item There does not exist an uncountable set $S \subset [a_{0}, b_{0}]$ such that neither
      player A nor player B has a winning strategy.
\end{enumerate}
\end{corollary}

\begin{proof}
All three of these assertions follow from the conclusion of Theorem~\ref{T:perf3} that
player A has a winning strategy for every uncountable set $S \subset [a_{0}, b_{0}]$.
\end{proof}

The Axiom of Determinacy includes the definition of a set being determined.  Since the
Axiom of Determinacy applies to infinite games defined in terms of collections of sequences
of natural numbers, and since the Cantor game involves increasing sequences of real numbers
and their limits, that definition is not directly applicable.  Nevertheless, it is
instructive to formulate an analogous definition for the Cantor game.

\begin{definition}\label{D:det}
Let $S \subset [a_{0}, b_{0}]$ with $a_{0} < b_{0}$.  We say that the set $S$ is
\emph{determined} if either player A or player B has a winning strategy in the Cantor game
with target set $S$.  Otherwise, the set $S$ is said to be \emph{non-determined}.
\end{definition}

According to Theorem~\ref{T:perf3} every uncountable subset of $[a_{0}, b_{0}]$ is
determined.  Since player B has a winning strategy for every countable subset of $[a_{0},
b_{0}]$ (see \cite[p.377]{mB07}), it follows that every subset of $[a_{0}, b_{0}]$ is
determined when we assume the axioms of ZF + AD.  Thus under the assumption of the Axiom of
Determinacy, our results offer a parallel to the statement of that axiom.

\section{The Cantor Game and the Axiom of Choice}\label{S:AC}
If we assume the Axiom of Choice instead of the Axiom of Determinacy in addition to the
Zermelo-Fraenkel axioms, we obtain an entirely different set of answers to Baker's three
questions.  Since our answers to two of these questions require the concept of a Bernstein
set of real numbers, we begin with some relevant definitions and results.  Throughout this
section we assume that the Axiom of Choice holds in addition to the Zermelo-Fraenkel
axioms.  Thus the results will be valid in ZFC.

\subsection{Bernstein sets}\label{SS:bern}
In Theorem 5.3 \cite[p.24]{jO80} Oxtoby defines the notion of a Bernstein set of real
numbers.  Here we extend his definition to subsets of $\mathbb{R}$.

\begin{definition}\label{D:bern}
Let $X \subset \mathbb{R}$, and let $X$ have the subspace topology inherited from
$\mathbb{R}$.  A \emph{Bernstein set} in $X$ is a set $S \subset X$ such that neither $S$
nor $X - S$ contains an uncountable closed set.
\end{definition}

Oxtoby also proves the following result.  His proof uses the Well-ordering Principle, which
is equivalent to the Axiom of Choice.

\begin{theorem}\label{T:bern1}
There exists a set $B$ of real numbers such that $B$ is a Bernstein set in $\mathbb{R}$.
\end{theorem}

Our first lemma provides us with a convenient method of making new Bernstein sets from
existing ones.

\begin{lemma}\label{L:bern}
Let $X$ and $Y$ be uncountable closed subsets of $\mathbb{R}$ with $Y \subset X$.  If $B$
is a Bernstein set in $X$, then $B' = B \cap Y$ is a Bernstein set in $Y$.
\end{lemma}

\begin{proof}
Since $Y$ is uncountable and closed, $B' \neq \emptyset$.  Let $C$ be any uncountable
closed set in $Y$.  Write $C = D \cap Y$, where $D$ is an uncountable closed set in $X$.
We see that $C$ is an uncountable closed set in $X$, so that $B \cap C \neq \emptyset$ and
$(X - B) \cap C \neq \emptyset$.  Observe that if $x \in B \cap C$, then $x \in Y$, and so
$x \in B' \cap C$.  Now if $y \in (X - B) \cap C$, then $y \in Y$.  Since $B' \subset B$
and $y \notin B$, we have $y \notin B'$.  Hence $y \in (Y - B') \cap C$.  This shows that
$B'$ is a Bernstein set in $Y$.
\end{proof}

Using these results we prove the existence of a Bernstein set in a closed interval.

\begin{corollary}\label{C:bern1}
There exists a Bernstein set $B'$ in the closed interval $[a, b]$, where $a < b$.
\end{corollary}

\begin{proof}
Suppose that $B$ is a Bernstein set in $\mathbb{R}$, and let $B' = B \cap [a, b]$.
Applying Lemma~\ref{L:bern} with $X = \mathbb{R}$ and $Y = [a, b]$ shows that $B'$ is a
Bernstein set in $[a, b]$.
\end{proof}

We now show that intersecting a Bernstein set in $[a, b]$ with a subinterval of $[a, b]$
yields a Bernstein set in that subinterval.  We shall need this property in a later
example.

\begin{corollary}\label{C:bern2}
Let $B'$ be a Bernstein set in the closed interval $[a, b]$, where $a < b$, and let
$[c, d]$ be any closed subinterval of $[a, b]$ with $a \le c < d \le b$.  Then $B'' = B'
\cap [c, d]$ is a Bernstein set in $[c, d]$.
\end{corollary}

\begin{proof}
This is a special case of Lemma~\ref{L:bern} with $X = [a, b]$ and $Y = [c, d]$.
\end{proof}

Bernstein sets have many interesting and pathological properties.  One property that will
be essential later is the following.

\begin{theorem}\label{T:bern2}
Let $B'$ be a Bernstein set in the closed interval $[a, b]$, where $a < b$.  Then $B'$ is
non-Lebesgue measurable and hence uncountable.
\end{theorem}

\begin{proof}
Our argument is similar to that given by Oxtoby \cite[p.24]{jO80} for his Bernstein set
$B \subset \mathbb{R}$.  Suppose, to arrive at a contradiction, that $B'$ is Lebesgue
measurable, and let $K$ be any compact subset of $B'$.  Since $K$ is closed, $K$ must be
countable, and so the Lebesgue measure $m(K) = 0$.  It follows that $m(B') = sup\{m(K) \mid
K \subset B', K\text{ compact}\} = 0$.  Since $[a, b] - B'$ is also Lebesgue measurable,
similar reasoning shows that $m([a, b] - B') = 0$.  This yields the contradiction that
$b - a = m([a, b]) = m(B') + m([a, b] - B') = 0$.  Thus $B'$ cannot be Lebesgue measurable.
Since every countable set is Lebesgue measurable and has measure 0, $B'$ must also be
uncountable.
\end{proof}

The following theorem will also be helpful.

\begin{theorem}\label{T:bern3}
Let $B'$ be a Bernstein set in the closed interval $[a, b]$, where $a < b$.  Then $B'$ and
$[a, b] - B'$ are dense in $[a, b]$.
\end{theorem}

\begin{proof}
By definition, $B'$ does not contain an uncountable closed set.  Since every perfect set is
uncountable and closed, $B'$ does not contain a perfect set.  By Lemma~\ref{L:bern3},
$[a, b] - B'$ is dense in $[a, b]$.  Similarly, $[a, b] - B'$ does not contain a perfect
set, and so by Lemma~\ref{L:bern3}, $[a, b] - ([a, b] - B') = B'$ is dense in $[a, b]$.
\end{proof}

\subsection{Examples}\label{SS:ex}
In this section we consider a pair of examples.  Our first example uses the following
theorem and its corollary, which is sometimes referred to as the Cantor-Bendixson Theorem.

\begin{theorem}\label{T:cb}
Let $S$ be an uncountable subset of $\mathbb{R}$.  Then there is a perfect set $P$ such
that $S \cap P$ is uncountable and $S - P$ is countable.
\end{theorem}

\begin{proof}
This is proved in Exercise 2.27 of \cite[p.45]{wR76}.  The proof shows that in fact $P =
C(S)$, the set of all condensation points of $S$.  See Definition~\ref{D:con} in Section~
\ref{SS:ZFC} for a definition of $C(S)$ and related sets.
\end{proof}

\begin{corollary}\label{C:cb}
Let $S$ be a closed and uncountable subset of $\mathbb{R}$.  Then there exists a perfect
set $P$ and a countable set $D$ such that $S = P \cup D$.
\end{corollary}

\begin{proof}
Let $P = C(S)$ be the set defined in Theorem~\ref{T:cb}, and let $D = S - P$.  Then $P$ is
a perfect set and $D$ is countable.  Since $S$ is closed, $P = C(S) \subset S$, and so $S =
P \cup (S - P) = P \cup D$.
\end{proof}

We are now ready for our first example, which illustrates Theorem~\ref{T:mainA}.

\begin{example}\label{E:a}
Consider the Cantor game on the interval $[0, 1]$ with target set $S = (\mathbb{R} -
\mathbb{Q}) \cap [0, 1]$; that is, $S$ is the set of irrational numbers in $[0, 1]$.  We
show that player A has a winning strategy.
\end{example}

\begin{proof}
Since $\mathbb{Q} \cap [0, 1]$ is countable, we may write $\mathbb{Q} \cap [0, 1] = \{q_{n}
\mid n \ge 1\}$.  For each $n\ge 1$ let
\[
   U_{n} = (q_{n} - \frac{1}{2^{n+2}}, q_{n} + \frac{1}{2^{n+2}}) \cap [0, 1]
\]
and define
\[
   U = \bigcup_{n \ge 1} U_{n}.
\]
We see that $m(U) \le \sum_{n=1}^{\infty} m(U_{n}) \le \sum_{n=1}^{\infty} \frac{1}
{2^{n+1}} = \frac{1}{2}$.  Now define $C = [0, 1] - U$.  Since $\mathbb{Q} \cap [0, 1]
\subset U$, $C \subset S$, and we have $m(C) = 1 - m(U) \ge \frac{1}{2}$.  Hence $C$ is
uncountable.  Since $U$ is the union of open sets in $[0, 1]$, $U$ is open, and so $C$ is
closed.  By Corollary~\ref{C:cb} we may write $C = P \cup D$, where $P$ is a perfect set
and $D$ is a countable set; then $P \subset C \subset S$.  Since $S$ contains the perfect
set $P$, player A has a winning strategy by Theorem~\ref{T:mainA}.
\end{proof}

Now the following example shows that the converse of Theorem~\ref{T:mainB} is not true.

\begin{example}\label{E:bern}
Let $S$ be a Bernstein set in $[0, \frac{1}{2}]$.  Now consider $S$ as a subset of
$[0, 1]$.  We show that the set $[0, 1] - S$ satisfies the three conditions of Theorem~
\ref{T:mainB} and that player B does not have a winning strategy when $S$ is the target set
for the Cantor game on $[0, 1]$.
\end{example}

\begin{proof}
We first observe that $[0, 1] - S = (\frac{1}{2}, 1] \cup ([0, \frac{1}{2}] - S)$ and that
both $S$ and $[0, \frac{1}{2}] - S$ do not contain perfect sets.  Since $(\frac{1}{2}, 1]$
contains the perfect set $[\frac{3}{4}, 1]$, $[0, 1] - S$ contains a perfect set.  If we
construct a Cantor middle-thirds set, which is a generalized Cantor set, inside
$[\frac{3}{4}, 1]$, then $[0, 1] - S$ contains this set.  Finally, since $S$ does not
contain a perfect set, it follows from Lemma~\ref{L:bern3} that $[0, 1] - S$ is dense in
$[0, 1]$.  Thus the three necessary conditions in Theorem~\ref{T:mainB} are satisfied.

To see that player B does not have a winning strategy, suppose that $(g_{n})_{n \ge 0}$ is
any strategy for player B.  Let $s \in S$ with $s > 0$ be given.  First define $a_{0} = 0$
and $b_{0} = 1$, and note that $b_{0} > s > 0$.  Next define $a_{1} = \frac{s}{2} = (1 -
\frac{1}{2^{1}}) s$ and
\[
   b_{1} =
   \begin{cases}
   0 &\text{if $g_{0}(a_{0}, b_{0}, a_{1}) < s$;}\\
   g_{0}(a_{0}, b_{0}, a_{1}) &\text{if $g_{0}(a_{0}, b_{0}, a_{1}) \ge s$.}
   \end{cases}
\]
Then define $a_{2} = \frac{3s}{4} = (1 - \frac{1}{2^{2}}) s$ and
\[
   b_{2} =
   \begin{cases}
   0 &\text{if $b_{1} = 0$;}\\
   0 &\text{if $b_{1} > 0$ and $g_{1}(a_{0}, b_{0}, a_{1}, b_{1}, a_{2}) < s$;}\\
   g_{1}(a_{0}, b_{0}, a_{1}, b_{1}, a_{2}) &\text{if $b_{1} > 0$ and $g_{1}(a_{0}, b_{0},
   a_{1}, b_{1}, a_{2}) \ge s$.}
   \end{cases}
\]
Now suppose that $n \ge 3$, that $a_{0}, b_{0}, a_{1}, b_{1}, \dots, a_{n-1}, b_{n-1}$ have
been defined, and that $a_{i} = (1 - \frac{1}{2^{i}}) s$ for $1 \le i \le n - 1$.  Define
$a_{n} = (1 - \frac{1}{2^{n}}) s$, $c_{n} = g_{n-1}(a_{0}, b_{0}, \dots, a_{n-1},
b_{n-1}, a_{n})$, and
\[
   b_{n} =
   \begin{cases}
   0 &\text{if $b_{n-1} = 0$;}\\
   0 &\text{if $b_{n-1} > 0$ and $c_{n} < s$;}\\
   c_{n} &\text{if $b_{n-1} > 0$ and $c_{n} \ge s$.}
   \end{cases}
\]
By induction we have defined sequences $(a_{n})_{n \ge 0}$ and $(b_{n})_{n \ge 0}$ such
that $a_{n} = (1 - \frac{1}{2^{n}}) s$ for all $n \ge 0$ and either (i) $b_{n} =
g_{n-1}(a_{0}, b_{0}, \dots, a_{n-1}, b_{n-1}, a_{n}) \ge s$ for all $n \ge 1$ or (ii)
there exists an integer $N \ge 0$ such that $b_{n} = g_{n-1}(a_{0}, b_{0}, \dots, a_{n-1},
b_{n-1}, a_{n}) \ge s$ for all $n \le N$ and $b_{n} = 0$ for all $n > N$.  We now show that
player B does not have a winning strategy in either case.

In case (i) we see that $((a_{n})_{n \ge 0}, (b_{n})_{n \ge 0})$ is a play of the Cantor
game on $[0, 1]$ which is consistent with the strategy of player B.  Since $\lim_{n \to
\infty} a_{n} = s \in S$, player A wins this particular game, and so $(g_{n})_{n \ge 0}$ is
not a winning strategy for player B.

Now consider case (ii).  In this case we define $\hat{a}_{0} = a_{N}$, $\hat{b}_{0} =
b_{N}$, and $\hat{S} = S \cap [\hat{a}_{0}, \hat{b}_{0}]$.  Observe that $\hat{S}$ is a
Bernstein set by Corollary~\ref{C:bern2}. For the Cantor game on $[\hat{a}_{0},
\hat{b}_{0}]$ with target set $\hat{S}$, define a new strategy for player B
$(\hat{g}_{n})_{n \ge 0}$ as follows.  Let $\hat{g}_{0}(\hat{a}_{0}, \hat{b}_{0},
\bar{a}_{1}) = g_{N}(a_{0}, b_{0}, a_{1}, b_{1}, \dots, a_{N}, b_{N}, \bar{a}_{1})$,
$\hat{g}_{1}(\hat{a}_{0}, \hat{b}_{0}, \bar{a}_{1}, \bar{b}_{1}, \bar{a}_{2}) =
g_{N+1}(a_{0}, b_{0}, a_{1}, b_{1}, \dots, a_{N}, b_{N}, \bar{a}_{1}, \bar{b}_{1},
\bar{a}_{2})$, and for $n \ge 3$, $\hat{g}_{n-1}(\hat{a}_{0}, \hat{b}_{0}, \bar{a}_{1},
\bar{b}_{1}, \dots, \bar{a}_{n-1}, \bar{b}_{n-1}, \bar{a}_{n}) = g_{N+n-1}(a_{0}, b_{0},
\dots, a_{N}, b_{N}, \bar{a}_{1}, \bar{b}_{1},\\ \dots, \bar{a}_{n-1}, \bar{b}_{n-1},
\bar{a}_{n})$ for all real numbers $\bar{a}_{1}, \dots, \bar{a}_{n}$ and $\bar{b}_{1},
\dots, \bar{b}_{n-1}$ satisfying $\hat{a}_{0} < \bar{a}_{1} < \dots < \bar{a}_{n} <
\bar{b}_{n-1} < \bar{b}_{n-2} < \dots < \hat{b}_{0}$.

Since $\hat{S}$ is a Bernstein set, the strategy $(\hat{g}_{n})_{n \ge 0}$ is not a winning
strategy for player B.  Let $((\hat{a}_{n})_{n \ge 0}, (\hat{b}_{n})_{n \ge 0})$ be a play
of the Cantor game on $[\hat{a}_{0}, \hat{b}_{0}]$ consistent with $(\hat{g}_{n})$ such
that player A wins.  Then $\hat{a} = \lim_{n \to \infty} \hat{a}_{n} \in \hat{S} \subset
S$.  Define $a_{n} = \hat{a}_{n-N}$ and $b_{n} = \hat{b}_{n-N}$ for all $n > N$.  By the
way that N was chosen, we have $b_{n} = g_{n-1}(a_{0}, b_{0}, \dots, a_{n-1}, b_{n-1},
a_{n})$ for all $n \le N$.  Now if $n > N$, then $b_{n} = \hat{b}_{n-N} =
\hat{g}_{n-N-1}(\hat{a}_{0}, \hat{b}_{0}, \hat{a}_{1}, \hat{b}_{1}, \dots, \hat{a}_{n-N-1},
\hat{b}_{n-N-1}, \hat{a}_{n-N}) = g_{n-1}(a_{0}, b_{0}, \dots, a_{N}, b_{N}, \hat{a}_{1},
\hat{b}_{1}, \dots, \hat{a}_{n-N-1}, \hat{b}_{n-N-1}, \hat{a}_{n-N}) = g_{n-1}(a_{0},
b_{0}, \dots,\\ a_{N}, b_{N}, a_{N+1}, b_{N+1}, \dots, a_{n-1}, b_{n-1}, a_{n})$.  Thus the
play of the game $((a_{n})_{n \ge 0},\\ (b_{n})_{n \ge 0})$ is consistent with the original
strategy $(g_{n})_{n \ge 0}$.  Since $\lim_{n \to \infty} a_{n} = \lim_{n \to \infty}
\hat{a}_{n} = \hat{a} \in S$, player A wins.  It follows that $(g_{n})_{n \ge 0}$ is not a
winning strategy for player B.
\end{proof}

\subsection{Answers in ZFC}\label{SS:ZFC}
We now show that in ZFC, where the concept of a Bernstein set is defined, the answer to
questions 1 and 3 is ``yes.''

\begin{corollary}\label{C:a}
Let $S$ be the target set for the Cantor game on $[a_{0}, b_{0}]$.  If $S$ is a Bernstein
set in $[a_{0}, b_{0}]$, then player A does not have a winning strategy.
\end{corollary}

\begin{proof}
Since every perfect set is both uncountable and closed, if $S$ is a Bernstein set, then $S$
does not contain a perfect set.  So by Theorem~\ref{T:mainA}, player A cannot have a
winning strategy in this case.
\end{proof}

Thus Corollary~\ref{C:a} shows that the answer to the first question is ``yes,'' and indeed
Theorem~\ref{T:mainA} characterizes all of the sets for which player A does not have a
winning strategy.  Now the following corollary shows that the answer to question 3 is also ``yes.''

\begin{corollary}\label{C:b}
Let $S$ be the target set for the Cantor game on $[a_{0}, b_{0}]$.  If $S$ is a Bernstein
set in $[a_{0}, b_{0}]$, then neither player A nor player B has a winning strategy.
\end{corollary}

\begin{proof}
We already proved in Corollary~\ref{C:a} that player A does not have a winning strategy.
Since every perfect set is both uncountable and closed, if $S$ is a Bernstein set, then the
complement of $S$ does not contain a perfect set.  So by Theorem~\ref{T:mainB}, player B
cannot have a winning strategy.
\end{proof}

Recalling our discussion of determinacy following Corollary~\ref{C:perf}, Corollary~
\ref{C:b} states that every Bernstein set in $[a_{0}, b_{0}]$ is non-determined when we
assume the axioms of ZFC.  Thus, under the assumption of the Axiom of Choice, our results
run parallel to results in descriptive set theory in which the Axiom of Choice contradicts
the Axiom of Determinacy.  Chapter 33 of Jech's monograph \cite{tJ03} contains an excellent
discussion of these results.

Example~\ref{E:bern} showed that the necessary conditions in Theorem~\ref{T:mainB} are not
sufficient to ensure that player B has a winning strategy.  We now investigate whether or
not player B can ever have a winning strategy.  To do so we consider the condensation
points of the target set $S$ and how these may relate to winning strategies for player B.

\begin{definition}\label{D:con}
Let $a_{0}, b_{0} \in \mathbb{R}$ with $a_{0} < b_{0}$, and let $S \subset [a_{0}, b_{0}]$.
Define
\[
   C_+(S) = \{x \in [a_{0}, b_{0}] \mid (x, x + \epsilon) \cap S\text{ is uncountable for
   every }\epsilon > 0\}
\]
to be the set of all right condensation points of $S$, and define
\[
   C_-(S) = \{x \in [a_{0}, b_{0}] \mid (x - \epsilon, x) \cap S\text{ is uncountable for
   every }\epsilon > 0\}
\]
to be the set of all left condensation points of $S$.  Then
\[
   C(S) = C_+(S) \cup C_-(S)
\]
is the set of all condensation points of $S$.

Now define the set of all two-sided condensation points of $S$ as
\[
   T(S) = C_+(S) \cap C_-(S);
\]
the set of left-only condensation points of $S$ as 
\[
   L(S) = C_-(S) - C_+(S);
\]
the set of right-only condensation points of $S$ as 
\[
   R(S) = C_+(S) - C_-(S);
\]
and the set of one-sided condensation points of $S$ as
\[
   O(S) = L(S) \cup R(S).
\]
\end{definition}

The following results will be helpful in connecting winning strategies for player B and the
condensation points of the set $S$.

\begin{theorem}\label{T:con1}
Let $S \subset [a_{0}, b_{0}]$, where $a_{0} < b_{0}$.  If $S$ is uncountable, then $S -
C_+(S)$ is countable, and so $S \cap C_+(S)$ and $C_+(S)$ are uncountable.
\end{theorem}

\begin{proof}
The proof of Exercise 2.27 in \cite[p.45]{wR76}, to which we referred in the proof of
Theorem~\ref{T:cb}, also implies that $L(S)$ and $R(S)$ are countable, so that $O(S) = L(S)
\cup R(S)$ is countable.  By Theorem~\ref{T:cb} $S \cap C(S)$ is uncountable and $S - C(S)$
is countable.  It follows that $S - T(S) = S - (C(S) - O(S)) = (S - C(S)) \cup (S \cap
O(S))$ is countable.  Now we have $T(S) \subset C_+(S)$, so that $S - C_+(S) \subset S -
T(S)$ is countable.  Since $S = (S \cap C_+(S)) \cup (S - C_+(S))$ and $S$ is uncountable,
$S \cap C_+(S)$ is uncountable, so that $C_+(S)$ is also uncountable.
\end{proof}

\begin{theorem}\label{T:15}
If $x \in S \cap C_+(S)$ and $x < y$, then there are uncountably many points $z$ in $S \cap
C_+(S)$ with $x < z < y$.
\end{theorem}

\begin{proof}
Suppose that $x \in S \cap C_+(S)$ and $y \in \mathbb{R}$ is such that $x < y$.  If
$\epsilon = y - x$, then $(x, x + \epsilon) = (x, y)$ contains uncountably many points of
$S$.  Since $S - C_+(S)$ is countable by Theorem~\ref{T:con1}, $(x, y)$ must contain
uncountably many points in $S \cap C_+(S)$.
\end{proof}

Our next result reveals the connection between the countability of the set $S$ and the set
of its condensation points.

\begin{theorem}\label{T:con2}
Let $S \subset [a_{0}, b_{0}]$, where $a_{0} < b_{0}$.  Then $S$ is countable if and only
if $C_+(S) = \emptyset$.
\end{theorem}

\begin{proof}
If $S$ is countable, then clearly $(x, x + \epsilon) \cap S$ cannot be uncountable for any
$x$; hence $C_+(S) = \emptyset$.  Now assume that $S$ is uncountable.  By Theorem~
\ref{T:con1}, $C_+(S)$ is uncountable, so that $C_+(S) \neq \emptyset$.
\end{proof}

\begin{corollary}\label{C:con}
If $C_+(S) = \emptyset$, then player B has a winning strategy.
\end{corollary}

\begin{proof}
If $C_+(S) = \emptyset$, then Theorem~\ref{T:con2} implies that $S$ is countable.  By
\cite[p.377]{mB07}, player B has a winning strategy.
\end{proof}

\begin{conjecture}
If $C_+(S) \neq \emptyset$, then player B does not have a winning strategy.
\end{conjecture}

Theorem~\ref{T:con2}, Corollary~\ref{C:con}, and this conjecture imply that player B has a
winning strategy if and only if $S$ is countable.  Thus if this conjecture is correct, then
the answer to question 2 is ``no'' when we assume the axioms of ZFC.

To see that this conjecture is very likely to be true, we begin by showing that player A
may choose the sequence $(a_{n})_{n \ge 0}$ so that for every $n \ge 1$ there are
uncountably many choices of $a_{n}$ from the set $S \cap C_+(S)$.

Let $S \subset [a_{0}, b_{0}]$, where $a_{0} < b_{0}$, be an uncountable set, and define
$c_{0} = \inf(S \cap C_+(S))$.  Then $a_{0} \le c_{0}$, and since $S \cap C_+(S)$ is
uncountable according to Theorem~\ref{T:con1}, $c_{0} < b_{0}$.  If $c_{0} \in S \cap
C_+(S)$, then by Theorem~\ref{T:15} there are uncountably many points $a_{1} \in S \cap
C_+(S)$ such that $a_{0} \le c_{0} < a_{1} < b_{0}$.  If $c_{0} \notin S \cap C_+(S)$, then
we may choose a point $c_{1} \in S \cap C_+(S)$ such that $c_{0} < c_{1} < b_{0}$.  By
Theorem~\ref{T:15} there are uncountably many points $a_{1} \in S \cap C_+(S)$ such that
$a_{0} \le c_{0} < c_{1} < a_{1} < b_{0}$.  In either case, assume that player A has
selected one of these points $a_{1}$.

Now suppose that $b_{1}$ has been chosen such that $a_{1} < b_{1} < b_{0}$.  Again by
Theorem~\ref{T:15}, since $a_{1} \in S \cap C_+(S)$, there are uncountably many points
$a_{2} \in S \cap C_+(S)$ with $a_{1} < a_{2} < b_{1}$.  Assume that player A has selected
one of these points $a_{2}$ and that player B has chosen $b_{2}$ such that $a_{2} < b_{2}
< b_{1}$.  Continuing this process we see that for each $n \ge 1$ player A has uncountably
many choices for $a_{n}$ from the set $S \cap C_+(S)$.

In order to show that player B does not have a winning strategy, we need only find a single
play of the Cantor game in which player A wins.  Recall that the winner of a play of the
Cantor game is determined by the value of $a = \lim_{n \to \infty} a_{n}$.  When $a \in S$,
player A wins, and when $a \in [a_{0}, b_{0}] - S$, player B wins.  When the sequence
$(a_{n})_{n \ge 0}$ is chosen so that $a_{n} \in S \cap C_+(S)$ for all $n \ge 1$, $a \in
\overline{S \cap C_+(S)} \subset \overline{S \cap C(S)} \subset \overline{C(S)} = C(S)$.
Since $S \cap C_+(S)$ is uncountable, $S \cap C(S)$ is also uncountable.  Thus there are
uncountably many ways that player A may choose each term of the sequence $(a_{n})$ and
uncountably many potential limits in $S \cap C(S)$ that will yield a play of the game that
player A wins.  Each time player B chooses a value for the next term $b_{n}$, player A's
choices are restricted, yet player A maintains having uncountably many choices from $S \cap
C_+(S)$ at each step.  So it seems very likely that at least one such sequence will result
in player A winning a game.

Recall from Theorem~\ref{T:mainB} that if player B has a winning strategy, then
$[a_{0}, b_{0}] - S$ is dense in $[a_{0}, b_{0}]$ and contains a perfect set.  In order for
player B to control the outcome of a play of the Cantor game, player B's chosen sequence
$(b_{n})_{n \ge 0}$ must force player A's sequence $(a_{n})_{n \ge 0}$ to converge to a
point of $[a_{0}, b_{0}] - S$.  This seems to require that $b = \lim_{n \to \infty} b_{n} =
\lim_{n \to \infty} a_{n} = a$ and that the points $b_{n}$ be chosen from a subset of
$[a_{0}, b_{0}] - S$ that is both closed and dense in some subinterval of $[a_{0}, b_{0}]$.
Thus $[a_{0}, b_{0}] - S$ would have to contain a closed subinterval that includes points
of player A's chosen sequences, but this is impossible when player A's sequences are chosen
from $S \cap C_+(S)$.

Although it may be possible to employ the Axiom of Choice directly to construct a play of
the Cantor game that player A wins, a more promising approach is to use Zorn's Lemma.  In
this approach we consider increasing sequences of real numbers in $[a_{0}, b_{0}]$ which
determine valid plays of the Cantor game on $[a_{0}, b_{0}]$, and we define a preordering
on some collection of such sequences.  If we can show that each chain of elements in that
collection has an upper bound which also lies in the collection, we can conclude from
Zorn's Lemma that this collection possesses a maximal element.  If we can then show that
this maximal element must have its limit in the set S, then we will have found a play of
the Cantor game that player A wins.  This will prove that player B never has a winning
strategy, and we believe that it will be a ``winning strategy'' for constructing a proof of
our conjecture.

\newpage

\renewcommand{\thepage}{}

\section{Acknowledgements}\label{S:ack}
I would like to thank Professor Matt Baker for proposing this project.  I would also like
to thank Professor Ted Slaman for his suggestions on what could be proved and how to get
started.  Finally I would like to thank Professor Baker again for giving my studies
direction and purpose for the last several years.

\end{document}